\documentclass[reqno,12pt]{amsart}
\usepackage{stmaryrd}
\usepackage{amsfonts}
\usepackage{bbm}
\usepackage{} 
\usepackage{appendix}

\numberwithin{equation}{section}
\usepackage{indentfirst}
 \usepackage{xcolor}

\usepackage{amssymb}
\usepackage{mathrsfs}
\usepackage{xy}
\xyoption{all}\newcommand{\qbinom}[2]{\textstyle\begin{bmatrix} #1\\#2 \end{bmatrix} }\newcommand{\qqbinom}[2]{\textstyle\llbracket\begin{smallmatrix} #1\\#2 \end{smallmatrix}\rrbracket }
\def\Ext{\mbox{\rm Ext}\,} \def\Hom{\mbox{\rm Hom}} \def\dim{\mbox{\rm dim}\,} \def\Iso{\mbox{\rm Iso}\,}\def\Ind{\mbox{\rm Ind}\,}
\def\lr#1{\langle #1\rangle}    
    
\def\End{\mbox{\rm End}\,}

\def\Dim{\mbox{\rm \textbf{dim}}\,}\def\A{\mathcal{A}\,} \def\H{\mathcal{H}\,}
\def\I{\mathcal{I}\,}
\def \I{\mathbb{I}}
\def \bv{\mathbf{v}}

\theoremstyle{plain} 
\newtheorem{theorem}{\bf Theorem}[section]
\newtheorem{lemma}[theorem]{\bf Lemma}
\newtheorem{corollary}[theorem]{\bf Corollary}
\newtheorem{proposition}[theorem]{\bf Proposition}

\theoremstyle{definition} 
\newtheorem{definition}[theorem]{\bf Definition}
\newtheorem{remark}[theorem]{\bf Remark}
\newtheorem{example}[theorem]{\bf Example}

\newcommand{\bt}{\begin{theorem}}
\newcommand{\et}{\end{theorem}}
\newcommand{\bl}{\begin{lemma}}
\newcommand{\el}{\end{lemma}}
\newcommand{\bd}{\begin{definition}}
\newcommand{\ed}{\end{definition}}
\newcommand{\bc}{\begin{corollary}}
\newcommand{\ec}{\end{corollary}}
\newcommand{\bp}{\begin{proof}}
\newcommand{\ep}{\end{proof}}
\newcommand{\bx}{\begin{example}}
\newcommand{\ex}{\end{example}}
\newcommand{\br}{\begin{remark}}
\newcommand{\er}{\end{remark}}
\newcommand{\be}{\begin{equation}}
\newcommand{\ee}{\end{equation}}
\newcommand{\ba}{\begin{align}}
\newcommand{\ea}{\end{align}}
\newcommand{\bn}{\begin{enumerate}}
\newcommand{\en}{\end{enumerate}}
\newcommand{\bcs}{\begin{cases}}
\newcommand{\ecs}{\end{cases}}

\makeatletter
\renewcommand{\section}{\@startsection{section}{1}{0mm}
  {-\baselineskip}{0.5\baselineskip}{\bf\leftline}}
\makeatother

\begin{document}

\title[From quantum groups to quantum cluster algebras]{From quantum groups to quantum cluster algebras} 

\author{Changjian Fu and Haicheng Zhang}
\address{Department of Mathematics\\ SiChuan University\\ Chengdu 610064, P.~R.~China}
\email{changjianfu@scu.edu.cn (C. Fu)}
\address{Ministry of Education Key Laboratory of NSLSCS, School of Mathematical Sciences, Nanjing Normal University, Nanjing 210023, P.R.China}
\email{zhanghc@njnu.edu.cn (H. Zhang)}

%


\begin{abstract}
We provide a homomorphism of algebras from the quantum group $\mathbf{U}^+_v(\mathfrak{g})$ to the corresponding quantum cluster algebra $\A_q$ with principal coefficients. As a by-product, we show that the quantum cluster variables arising from
one-step mutations from the initial cluster variables satisfy the (high order) quantum Serre relations in $\A_q$.
\end{abstract}

\maketitle

\section{Introduction}
The Ringel--Hall algebra has established a bridge between representation theory of algebras and Lie theory. The Ringel--Hall algebra of a finite-dimensional hereditary algebra
provides a realization of the half part of the corresponding quantum group (cf.  \cite{R90,R90a,Gr95}).

The cluster algebra was invented by Fomin and Zelevinsky \cite{FZ} with the aim to study the total positivity in algebraic groups and canonical bases of quantum groups.
The connections between the cluster algebra and representation theory are explicitly characterized by cluster characters, which were defined by Caldero and Chapoton in \cite{CC}. The Caldero--Keller cluster multiplication theorem of cluster characters in \cite{CK2005,CK2} revealed the similarity between the
multiplications in cluster algebras and those in dual Hall algebras.

As the quantization of cluster algebras, quantum cluster algebras were later introduced by Berenstein and Zelevinsky in \cite{BZ05}. Berenstein and Rupel \cite{BR} introduced a generalized quantum cluster character, which assigns to each object $V$ of a hereditary abelian category $\mathcal {C}$ and any sequence ${\bf i}$ of simple objects in $\mathcal {C}$ the element $X_{V,{\bf i}}$ of a certain quantum polynomial algebra $\mathcal{P}_{\mathcal {C},{\bf i}}$. They defined an algebra homomorphism from the dual Ringel--Hall algebra of $\mathcal {C}$ to $\mathcal{P}_{\mathcal {C},{\bf i}}$, which generalizes the well-known Feigin homomorphism from the half part of the corresponding quantum group to $\mathcal{P}_{\mathcal {C},{\bf i}}$. The quantum versions of cluster characters defined by Caldero and Chapoton, the so-called quantum cluster characters were introduced by Rupel \cite{Rupel1}. Recently, by using the integration homomorphism on the Hall algebra of the morphism category of projective modules over a hereditary algebra, as well as its bialgebra structure, Fu, Peng and Zhang~\cite{FPZ} provided an intrinsic realization of quantum cluster characters via Hall algebra approach. Subsequently, Chen, Ding and Zhang \cite{CDZ} developed the approach used in~\cite{FPZ} and obtained the quantum version of Caldero--Keller's cluster multiplication theorem.

Although there have been homomorphisms of algebras from dual Hall algebras of hereditary abelian categories to the corresponding quantum cluster algebras (cf. \cite{BR,CDX,DXZ,FPZ}), these Hall algebras are all needed to be twisted by using the Euler form and the skew-symmetric bilinear form $\Lambda$ arising from a compatible pair.

Recently, Huang, Chen, Ding and Xu \cite{HCDX25} proved that, for an arbitrary quantum cluster algebra $\A_q$ with principal coefficients, the quantum cluster variables arising from one-step mutations from the initial cluster variables satisfy the so-called (high order) fundamental relations. Their approach is to do direct, elementary, but complicated calculations in $\A_q$.

In this paper, let $Q$ be a finite valued quiver without loops or two-cycles, and $\A$ be the category of finite-dimensional nilpotent representations of $Q$ over a finite field. Firstly, using the algebra homomorphism in \cite{CDX} for the compatible pair considered in \cite{HCDX25}, we recover the (high order) fundamental relations given in \cite{HCDX25}. Then we give another compatible pair $(\Lambda, \widetilde{B})$ associated to an acyclic quiver $Q$. Also, using the algebra homomorphism in \cite{CDX} for this special $\Lambda$, we obtain a homomorphism of algebras from the Ringel--Hall algebra $\mathcal {H}_v(\A)$ to the quantum cluster algebra $\A_q(Q)$ associated to $(\Lambda, \widetilde{B})$. Thus, we obtain the (high order) quantum Serre relations in the quantum cluster algebra $\A_q(Q)$.

Let us fix some notations used throughout the paper. For a finite set $S$, we denote by $|S|$ its cardinality. Let $k=\mathbb{F}_q$ be a finite field with $q$ elements, and set $v=\sqrt{q}$. Let $\A$ be a finitary hereditary abelian $k$-category, and let $\Iso(\mathcal {A})$ a complete set of objects in $\mathcal {A}$, $\Ind(\mathcal {A})$ be a complete set of indecomposable objects in $\mathcal {A}$, up to isomorphisms. The Grothendieck group of $\mathcal {A}$ is denoted by $K(\mathcal {A})$. Assume that all the vectors are column vectors, and denote by $X^{\rm tr}$ the transpose of a matrix $X$. For each finite-dimensional module $M$, we will always use the
corresponding lowercase boldface letter ${\bf m}$ to denote its dimension vector.

\section{Preliminaries}
\subsection{Quantum groups}
Let $\I=\{1,2,\dots,n\}$ be the index set and $C=(c_{ij})_{i,j \in \I}$ be a symmetrizable generalized Cartan matrix with a symmetrizer $D={\rm diag}(d_i\mid i\in\I)$, where all $d_i$ are chosen to be positive integers. That is, $DC$ is symmetric. Denote by $\mathfrak{g}$ the Kac--Moody Lie algebra corresponding to $C$.
Given two elements $a,b$ in an algebra, we write $[a, b]=ab-ba$. Let ${\bf v}$ be an indeterminate and $\mathbb{Q}({\bf v})$ be the rational function field.  Set ${\bf v}_i={\bf v}^{d_i}$ for each $i\in\I$. For any positive integer $a$, define
$[a] =\frac{{\bf v}^a-{\bf v}^{-a}}{{\bf v}-{\bf v}^{-1}}.$
For any positive integer $m$,  define $[m]^! =\prod\limits_{i=1}^m [i]$ and write $[0]^!=1$.
For any positive integers $d\leq m$, define
\begin{align*}
\qbinom{m}{d}  &=
\begin{cases}
\frac{[m][m-1]\ldots [m-d+1]}{[d]^!}, & \text{ if }d > 0,
\\
1,& \text{ if }d=0.
\end{cases}
\end{align*}
If $c$ is an integral power of ${\bf v}$, we write $\qbinom{m}{d}_c$ for the evaluation of $\qbinom{m}{d}$ at ${\bf v}=c$.

The \emph{quantum group} $\mathbf{U}_\bv(\mathfrak{g})$ is defined to be the $\mathbb{Q}(\bv^{\frac{1}{2}})$-algebra generated by $E_i,F_i, K_i,K_i^{-1}$, $i\in \I$,
subject to the following relations {for $i,j\in\I$}:
\begin{align}
&K_iK_i^{-1}=K_i^{-1}K_i=1,~\quad  [K_i,K_j]=0,\\
&K_i E_j=\bv_i^{c_{ij}} E_j K_i,\quad\qquad K_i F_j=\bv_i^{-c_{ij}} F_j K_i,\\
&[E_i,F_i]=\frac{K_i-K_i^{-1}}{\bv_i-\bv_i^{-1}},~~~~~~\quad[E_i,F_j]=0~~\text{for}~~i\neq j
\end{align}
and the {\em quantum Serre relations}
\begin{align}
&\sum_{t=0}^{1-c_{ij}} (-1)^t \qbinom{1-c_{ij}}{t}_{{\bf v}_i} E_i^{1-c_{ij}-t} E_j  E_i^{t}=0,\quad i\neq j\\
&\sum_{t=0}^{1-c_{ij}} (-1)^t \qbinom{1-c_{ij}}{t}_{{\bf v}_i} F_i^{1-c_{ij}-t} F_j  F_i^{t}=0,\quad i\neq j.
\end{align}

\subsection{Hall algebras}
Given objects $L,M,N \in \mathcal{A}$, let $\Ext_\mathcal{A}^1(M,N)_L \subset \Ext_\mathcal{A}^1(M,N)$ be the subset consisting of those equivalence classes of short exact sequences with middle term isomorphic to $L$.
\begin{definition}\label{Hall algebra of abelian category}
The \emph{Hall algebra} $\mathcal {H}(\mathcal{A})$ of $\mathcal{A}$ is the $\mathbb{Q}(v^{\frac{1}{2}})$-space with basis elements $\{u_M~|~ M\in \Iso(\mathcal{A})\}$, and with the multiplication defined by
\[u_M \diamond u_N = \sum\limits_{L \in {\rm Iso}(\mathcal{A})} {\frac{{|\Ext_\mathcal{A}^1{{(M,N)}_L}|}}{{|\Hom_\mathcal{A}(M,N)|}}} u_L.\]
\end{definition}

For any objects $M,N \in \mathcal{A}$, define \begin{equation}\label{Euler form}\lr{M,N}:=\dim_k\Hom_{\A}(M,N)-\dim_k\Ext^1_{\A}(M,N),\end{equation}
which descends to give a bilinear form
\begin{equation}\label{adeuler}\lr{\cdot ,\cdot }: K(\mathcal{A})\times K(\mathcal{A})\longrightarrow \mathbb{Z},\end{equation}
called the \emph{Euler form} of $\mathcal{A}$.
We also consider the \emph{symmetric Euler form}
$$(\cdot ,\cdot ): K(\mathcal{A})\times K(\mathcal{A})\longrightarrow \mathbb{Z},$$ defined by $(\alpha,\beta)=\lr{\alpha,\beta}+\lr{\beta,\alpha}$ for any $\alpha,\beta \in K(\mathcal{A})$.

The $v$-twisted Hall algebra $\mathcal {H}_{v}(\A)$, called the \emph{Ringel--Hall algebra} of $\A$, is the same space as $\mathcal {H}(\A)$, but with the twisted multiplication defined by
$$u_M*u_N=v^{\lr{M,N}}u_M\diamond u_N.$$

Let $Q$ be a valued quiver (cf. \cite{Rupel1,Rupel2}) without loops and its vertex set be $\{1,2,\cdots,n\}$. For each vertex $i$, let $d_i\in\mathbb{N^+}$ be the corresponding valuation. Let $\A$ be the category of finite-dimensional nilpotent $k$-representations of $Q$. Set  $\mathcal {D}_i=\End_{{\A}}(S_i)$ for each $1\leq i\leq n$, where $S_i$ is the simple representation corresponding to the vertex $i$.
For any $1\leq i,j\leq n$, set
$$c'_{ij}=\dim_{\mathcal {D}_i}\Ext_{{\A}}^1(S_j,S_i),\quad\quad c''_{ij}=\dim_{{\mathcal {D}_i}^{op}}\Ext_{{\A}}^1(S_i,S_j),$$
and define
\[c_{ij}=\begin{cases}
2 & \text{if $i=j$;}\\
-c'_{ij}-c''_{ij} & \text{if $i\neq j$.}
\end{cases}\]
Then the matrix $C=(c_{ij})$ is a
symmetrizable generalized Cartan matrix with symmetrizer $D={\rm diag}(d_1,\cdots,d_n)$. Denote by $\mathbf{U}_v(\mathfrak{g})$ the quantum group corresponding to $C$,  specialised at $\bv=v$. Let $\mathbf{U}^+_v(\mathfrak{g})$ be the positive part of $\mathbf{U}_v(\mathfrak{g})$ generated by all $E_i$.

In what follows, for each $1\leq i\leq n$, we may write $v_i$ and $q_i$ for $v^{d_i}$ and $q^{d_i}$, respectively.
\begin{theorem}\label{ringel}~\textup{(Ringel, Green)}
There exists an injective homomorphism of algebras $\lambda:\mathbf{U}^+_v(\mathfrak{g})\rightarrow \mathcal {H}_{v}(\A)$ defined on generators by
$\lambda(E_i)=(q_i-1)^{-1}u_{S_i}$. In particular, we have the quantum Serre relations
$$\sum_{t=0}^{1-c_{ij}} (-1)^t \qbinom{1-c_{ij}}{t}_{{v}_i} u_{S_i}^{1-c_{ij}-t} u_{S_j}  u_{S_i}^{t}=0,\quad 1\leq i\neq j\leq n$$
in the Ringel--Hall algebra $\mathcal {H}_{v}(\A)$.
\end{theorem}
In fact, we also have the following so-called {\em high order quantum Serre relations} in the Ringel--Hall algebra $\mathcal {H}_{v}(\A)$ (cf. \cite[Ch. 7]{Lus}, \cite{Deng}).
\begin{theorem}\label{high}
Let $1\leq i\neq j\leq n$. For any positive integers $l,p$ with $p\geq-lc_{ij}$, and $\varepsilon=\pm1$, we have
\begin{equation}\label{gjsgx}
\sum_{t=0}^{p+1} (-1)^t v_i^{\varepsilon (p+lc_{ij})t}\qbinom{p+1}{t}_{{v}_i} u_{S_i}^{p+1-t} u_{S_j}^l  u_{S_i}^{t}=0
\end{equation}
in the Ringel--Hall algebra $\mathcal {H}_{v}(\A)$.
\end{theorem}
Note that if we take $l=1$ and $p=-c_{ij}$ in Theorem \ref{high}, the corresponding high order quantum Serre relations are just quantum Serre relations.

\subsection{Quantum cluster algebras}
In what follows, let $n\leq m$ be always two positive integers. For each integer $a$, set $[a]_{+}:=max\{0,a\}$.
Let $\Lambda$ be an $m\times m$ skew-symmetric matrix {such that $2\Lambda$ is an integral matrix}, and denote by $\{{\bf e}_1,\cdots,{\bf e}_m\}$ the standard basis of $\mathbb{Z}^{m}$. Let $\mathfrak{q}$ be an indeterminate.
Define the {\em quantum torus} associated to the pair
$(\mathbb{Z}^{m},\Lambda)$ to be the $\mathbb{Q}(\mathfrak{q}^{\frac{1}{4}})$-algebra $\mathcal{T}_{\mathfrak{q},\Lambda}$ with
a distinguished basis $\{X^{\bf e}: {\bf e}\in \mathbb{Z}^{m}\}$ and
multiplication given by
\[X^{\bf e}X^{\bf f}=\mathfrak{q}^{\frac{1}{2}\Lambda({\bf e},{\bf f})}X^{{\bf e+f}},\]
where we still denote by $\Lambda$ the skew-symmetric bilinear form on $\mathbb{Z}^{m}$ associated to the skew-symmetric matrix $\Lambda$. It is well-known that $\mathcal{T}_{\mathfrak{q},\Lambda}$ is an Ore domain, and thus is contained in its
skew-field of fractions $\mathcal{F}_{\mathfrak{q},\Lambda}$.

An $n\times n$ integral matrix $B$ is called {\em skew-symmetrizable}, if there exists a
diagonal matrix $D={\rm diag}(d_1,\cdots,d_n)$ with each $d_i\in\mathbb{N^+}$, such that $DB$ is skew-symmetric. Let $\widetilde{B}=(b_{ij})$ be an $m\times n$ integral matrix with the upper
$n\times n$ submatrix $B$ being skew-symmetrizable, and $\Lambda$ is as above, then the pair
$(\Lambda, \widetilde{B})$ is called {\em compatible} if $\widetilde{B}^{\rm tr}\Lambda=(D~|~0)$ for some diagonal matrix
$D={\rm diag}(d_1,\cdots,d_n)$ with each $d_i\in\mathbb{N^+}$. An {\em initial  quantum seed} for $\mathcal{F}_{\mathfrak{q},\Lambda}$ is a triple
$(\Lambda, \widetilde{B}, {\bf \widetilde{x}})$  consisting of a compatible pair $(\Lambda,\widetilde{B})$ and the set  ${\bf \widetilde{x}}=\{x_1,\cdots,x_m\}$, where each $x_i$ denotes $X^{{\bf e}_i}$. For any $1\leq k\leq n$, we  define the mutation of $(\Lambda, \widetilde{B}, {\bf \widetilde{x}})$ in the direction $k$ to obtain the new quantum seed $(\Lambda',\widetilde{B}',{\bf \widetilde{x}}')$ as follows:

(1)\ $\Lambda'=H^{\rm tr}\Lambda H$, where the
$m\times m$ matrix $H=(h_{ij})$ is given by
\[h_{ij}=\begin{cases}
\delta_{ij} & \text{if $j\ne k$;}\\
-1 & \text{if $i=j=k$;}\\
[-b_{ik}]_+ & \text{if $i\ne j = k$.}
\end{cases}
\]

(2)\ $\widetilde{B}'=(b'_{ij})$ is given by
\[b'_{ij}=\begin{cases}
-b_{ij} & \text{if $i=k$ or $j=k$;}\\
b_{ij}+\frac{1}{2}(|b_{ik}|b_{kj}+b_{ik}|b_{kj}|) & \text{otherwise.}
\end{cases}
\]

(3)\  ${\bf \widetilde{x}}'=\{x'_1,\cdots,x'_m\}$ is given by
\begin{equation}\label{mutation}
x_i'=\begin{cases}
X^{\sum\limits_{i=1}^m[b_{ik}]_{+} {\bf e}_i -{\bf e}_k}+X^{\sum\limits_{i=1}^m[-b_{ik}]_{+} {\bf e}_i -{\bf e}_k} & \text{if $i= k$;}\\
x_i & \text{otherwise.}
\end{cases}
\end{equation}

Two quantum seeds  $(\Lambda, \widetilde{B}, {\bf \widetilde{x}})$ and  $(\Lambda', \widetilde{B}', {\bf \widetilde{x}}')$ are
called {\em mutation-equivalent}, denoted by $(\Lambda, \widetilde{B}, {\bf \widetilde{x}})\sim(\Lambda', \widetilde{B}', {\bf \widetilde{x}}')$, if they can be obtained from each other
by a sequence of mutations. Let $\mathcal{C}=\{x'_i~|~ (\Lambda', \widetilde{B}', {\bf \widetilde{x}}')\sim(\Lambda, \widetilde{B}, {\bf \widetilde{x}}),1\leq i\leq n\}$, and
the elements in $\mathcal{C}$ are called the {\em quantum cluster
variables}. Let $\mathbb{P}=\{x_i~|~n+1\leq i\leq m\}$, and the
elements in $\mathbb{P}$ are called the {\em frozen variables}. Denote by
$\mathbb{Q}\mathbb{P}$ the algebra of Laurent polynomials in the elements of $\mathbb{P}$ with coefficients in $\mathbb{Q}(\mathfrak{q}^{\frac{1}{4}})$. Then the
{\em quantum cluster algebra}
$\mathcal{A}_{\mathfrak{q}}(\Lambda,\widetilde{B})$ is defined to be the
$\mathbb{Q}\mathbb{P}$-subalgebra of $\mathcal{F}_{\mathfrak{q},\Lambda}$ generated by all
quantum cluster variables.

\section{Twisted high order quantum Serre relations in $\mathcal{T}_{q,\Lambda}$}
\subsection{Setting and Notation}\label{s3.1}
Let $\widetilde{Q}$ be a valued quiver without loops or two-cycles, and its vertex set be
$\{1,\ldots,m\}$. Let $Q$ be a full subquiver of $\widetilde{Q}$ with the vertex set $\{1,\ldots,n\}$. Denote by $\mathcal{A}$ and $\widetilde{\mathcal{A}}$ be the categories of finite-dimensional nilpotent $k$-representations of $Q$ and $\widetilde{Q}$, respectively. For each $1\leq i\leq m$, let $S_i$ be the simple representation of $\widetilde{Q}$ corresponding to the vertex $i$, and denote by $\mathcal{D}_i=\End(S_i)$ its endomorphism algebra. Set $D=\operatorname{diag}(d_1,\ldots,d_n)$,  where each $d_i=\dim_k\mathcal{D}_i$.

Let $\widetilde{R}$ and $\widetilde{R}'$ be the $m\times n$ matrices with the $i$-th row and $j$-th column elements given respectively by
\[
r_{ij}=\dim_{\mathcal{D}_i}\Ext^1_{\widetilde{\mathcal{A}}}(S_j,S_i)
\]
and
\[
r_{ij}'=\dim_{\mathcal{D}_i^{op}}\Ext^1_{\widetilde{\mathcal{A}}}(S_i,S_j),
\]
where $1\leq i\leq m$ and $1\leq j\leq n$. Set $\widetilde{E}'=\widetilde{I}-\widetilde{R}$ and $\widetilde{E}=\widetilde{I}-\widetilde{R}'$, where $\widetilde{I}$ is the left $m\times n$ submatrix of the $m\times m$ identity matrix $I_m$. Recall that $c_{ij}=-r_{ij}-r_{ij}'$ for any $1\leq i\neq j\leq n$.

Let $R,R'\in M_{n\times n}(\mathbb{Z})$ be the upper $n\times n$ sub-matrices of $\widetilde{R}$ and $\widetilde{R}'$, respectively.
Set $\widetilde{B}=\widetilde{R}'-\widetilde{R}$ and $B=R'-R$. Since $DR'=R^{\rm tr}D$,
$DB$ is skew-symmetric. Moreover, the matrix $E$ representing the Euler form of $\A$ under the standard basis is equal to $(I_n-R^{\rm tr})D=D(I_n-R')$.

In what follows, we assume that there is a compatible pair $(\Lambda, \widetilde{B})$ with
$\widetilde{B}^{\rm tr}\Lambda=(D,0)$.

\subsection{$\Lambda$-twisted Ringel--Hall algebras}
The {\em $\Lambda$-twisted Ringel--Hall algebra} $\mathcal {H}_{\Lambda}(\A)$ is the same space as the Ringel--Hall algebra $\mathcal {H}_v(\A)$, but with the twisted multiplication defined by
\begin{equation}\label{lmdc}u_M\star u_N=v^{\Lambda(\widetilde{E}{\bf m},\widetilde{E}{\bf n})+\lr{M,N}}u_M\ast u_N
=q^{\frac{1}{2}\Lambda(\widetilde{E}{\bf m},\widetilde{E}{\bf n})+\lr{M,N}}u_M\diamond u_N.\end{equation}

For any object $M\in\A$, set
$$X_M:=\sum\limits_{\mathbf{e}}v^{-\lr{\mathbf{e}, \mathbf{m}-\mathbf{e}}}|\mathrm{Gr}_{\mathbf{e}}M|
X^{-\widetilde{E}'\mathbf{e}-\widetilde{E}(\mathbf{m}-\mathbf{e})},$$
where $\mathrm{Gr}_{\bf{e}}M$ is the set of all submodules $V$ of $M$ with $\Dim V= \bf{e}$.

\begin{remark}
When the valued quiver $Q$ is acyclic, $X_?$ is known as the quantum Caldero--Chapoton map, which provides a representation-theoretic realization of quantum cluster variables. In particular, for every rigid representation $M$, $X_M$ lies in the quantum cluster algebra $\mathcal{A}_q(\Lambda,\widetilde{B})$, and in fact corresponds to a quantum cluster monomial. However, when $Q$ contains oriented cycles,
it is unclear whether $X_M$
lies in the quantum cluster algebra  for an arbitrary  rigid representation $M$ of $Q$.
\end{remark}

In what follows, for each $1\leq i\leq n$, we write $y_i$ for the quantum cluster variable $x'_i$ arising from
one-step mutation in the direction $i$ from the initial cluster ${\bf \widetilde{x}}$.
\begin{proposition}\label{sy2}
For each $1\leq i\leq n$, we have $y_i=X_{S_i}$.
\end{proposition}
\begin{proof}
By definition, for each $1\leq i\leq n$, $X_{S_i}=X^{-\widetilde{E}{\bf e}_i}+X^{-\widetilde{E}'{\bf e}_i}$.
Since $b_{kj}=r'_{kj}-r_{kj}$ for any $1\leq k\leq m$ and $1\leq j\leq n$, we have $$r'_{kj}=[b_{kj}]_{+}\quad\text{and}\quad r_{kj}=[-b_{kj}]_{+}.$$
So the $i$-th column vector of $\widetilde{E}$ is equal to
\begin{flalign*}\widetilde{E}{\bf e}_i&=(-r'_{1i},\cdots,-r'_{i-1,i},1,-r'_{i+1,i},\cdots,-r'_{m,i})^{\rm tr}\\
&=(-[b_{1i}]_+,\cdots,-[b_{i-1,i}]_+,1,-[b_{i+1,i}]_+,\cdots,-[b_{m,i}]_+)^{\rm tr}.
\end{flalign*}
Hence, we obtain $X^{-\widetilde{E}{\bf e}_i}=X^{\sum\limits_{k=1}^{m}[b_{ki}]_+{\bf e}_k-{\bf e}_i}$. Similarly,
$X^{-\widetilde{E}'{\bf e}_i}=X^{\sum\limits_{k=1}^{m}[-b_{ki}]_+{\bf e}_k-{\bf e}_i}$.
Therefore, $$X_{S_i}=X^{\sum\limits_{k=1}^{m}[b_{ki}]_+{\bf e}_k-{\bf e}_i}+X^{\sum\limits_{k=1}^{m}[-b_{ki}]_+{\bf e}_k-{\bf e}_i}=y_i.$$
\end{proof}

The following homomorphism of algebras has been proved in \cite{CDX} for acyclic quivers, but it also holds for any valued quiver which may have oriented cycles (cf. also \cite{BR,FPZ}).
In fact, according to \cite{FPZ}, this homomorphism can also be obtained by using the comultiplication homomorphism and integration homomorphism on the Hall algebra of $\A$, which only depend on the heredity of $\A$.
\begin{theorem}\label{alg-hom-cdx1}
The map $\Psi:\mathcal {H}_{\Lambda}(\A)\longrightarrow\mathcal{T}_{q,\Lambda}$ defined on basis elements by $u_M\mapsto X_M$ is a homomorphism of algebras.
\end{theorem}

\subsection{Twisted high order quantum Serre relations}
Let $1\leq i\neq j\leq n$, for any positive integers $p,l$ with $p\geq-lc_{ij}$, and $0\leq t\leq p+1$, set
\begin{flalign*}
s_{i,j;t}^{p,l}&:=\frac{(p+1-t)(p-t)}{2}d_i+\frac{l(l-1)}{2}d_j+\frac{t(t-1)}{2}d_i+(p+1-t)l\Lambda(\widetilde{E}{\bf e}_i,\widetilde{E}{\bf e}_j)\\
&+(p+1-t)l\lr{{\bf e}_i,{\bf e}_j}+lt\Lambda(\widetilde{E}{\bf e}_j,\widetilde{E}{\bf e}_i)+(p+1-t)td_i+lt\lr{{\bf e}_j,{\bf e}_i}\\
&=\frac{p(p+1)}{2}d_i+\frac{l(l-1)}{2}d_j+(p+1-t)l\lr{{\bf e}_i,{\bf e}_j}+lt\lr{{\bf e}_j,{\bf e}_i}+(p+1-2t)l\Lambda(\widetilde{E}{\bf e}_i,\widetilde{E}{\bf e}_j)
\end{flalign*}
and $$a_{i,j;l,t}:=(\lr{{\bf e}_j,{\bf e}_i}-\lr{{\bf e}_i,{\bf e}_j}-2\Lambda(\widetilde{E}{\bf e}_i,\widetilde{E}{\bf e}_j))lt.$$

\begin{proposition}\label{eq:twist-q-serre-rel-qca}
Let $1\leq i\neq j\leq n$. For any positive integers $l,p$ with $p\geq-lc_{ij}$, and $\varepsilon=\pm1$, we have the twisted high order quantum Serre relations
$$\sum_{t=0}^{p+1} (-1)^t v_i^{\varepsilon (p+lc_{ij})t}v^{-a_{i,j;l,t}}\qbinom{p+1}{t}_{{v}_i} y_i^{p+1-t} y_j^l  y_i^{t}=0$$
in the quantum torus $\mathcal{T}_{q,\Lambda}$.
\end{proposition}
\begin{proof}
By rewriting the high order quantum Serre relation $(\ref{gjsgx})$ in $\mathcal {H}_{\Lambda}(\A)$, and using Theorem \ref{alg-hom-cdx1} and Proposition $\ref{sy2}$, we obtain
$$\sum_{t=0}^{p+1} (-1)^t v_i^{\varepsilon (p+lc_{ij})t}v^{-s_{i,j;t}^{p,l}}\qbinom{p+1}{t}_{{v}_i} y_{i}^{p+1-t} y_{j}^l  y_{i}^{t}=0$$
in the quantum torus $\mathcal{T}_{q,\Lambda}$. Then eliminating the constant multiple from the summation, we finish the proof.

\end{proof}

For any $1\leq i\neq j\leq n$ and $0\leq t\leq 1-c_{ij}$,
set \begin{flalign*}w_{i,j;t}:=a_{i,j;1,t}=(\lr{{\bf e}_j,{\bf e}_i}-\lr{{\bf e}_i,{\bf e}_j}-2\Lambda(\widetilde{E}{\bf e}_i,\widetilde{E}{\bf e}_j))t.\end{flalign*}
Then we have the following.
\begin{corollary}
For any $1\leq i\neq j\leq n$, we have the twisted quantum Serre relations
$$\sum_{t=0}^{1-c_{ij}} (-1)^tv^{-w_{i,j;t}}\qbinom{1-c_{ij}}{t}_{{v}_i} y_i^{1-c_{ij}-t} y_j  y_i^{t}=0$$
in the quantum torus $\mathcal{T}_{q,\Lambda}$.
\end{corollary}

\section{Fundamental relations in quantum cluster algebras}
In this section, let $B=(b_{ij})\in M_{n\times n}(\mathbb{Z})$ be a skew-symmetrizable matrix with a skew-symmetrizer $D=\operatorname{diag}(d_1,\ldots,d_n)$. Let $Q$ be the valued quiver associated with $(B,D)$.  Denote by $\widetilde{Q}$ the valued quiver obtained from $Q$ by attaching additional vertices $n+1,\ldots, 2n$ to $Q$, adding the arrow $n+i\to i$, and setting the valuation of the $(n+i)$-th vertex to be $d_{n+i}=d_i$ for each $1\leq i\leq n$.
As before, we have the $2n\times n$-matrices $\widetilde{R}={R\choose 0}, \widetilde{R}'={R'\choose I_n}$, $\widetilde{E}'=\widetilde{I}-\widetilde{R}$, $\widetilde{E}=\widetilde{I}-\widetilde{R}'$, and $\widetilde{B}=\widetilde{R}'-\widetilde{R}={B\choose I_n}$.

Let $\Lambda_1=\left(\begin{smallmatrix}
    0&-D\\ D&-DB
\end{smallmatrix}\right)$. Then $(\widetilde{B},\Lambda_1)$ is a compatible pair such that $\widetilde{B}^{\rm tr}\Lambda_1=(D,0)$.
\begin{lemma}\label{zbyl1}
For any $\alpha,\beta\in\mathbb{Z}^n$, we have
$\Lambda_1(\widetilde{E}\alpha,\widetilde{E}\beta)=0.$
\end{lemma}
\begin{proof}
Noting that $\widetilde{E}={I_n-R'\choose-I_n}={D^{-1}E\choose-I_n}$, we have
\begin{flalign*}
\Lambda_1(\widetilde{E}\alpha,\widetilde{E}\beta)&=\alpha^{\rm tr}(E^{\rm tr}D^{-1},-I_n)\left(\begin{smallmatrix}
    0&-D\\ D&-DB
\end{smallmatrix}\right)\left(\begin{smallmatrix}
    D^{-1}E\\ -I_n
\end{smallmatrix}\right)\beta\\
&=\alpha^{\rm tr}(-D,DB-E^{\rm tr})\left(\begin{smallmatrix}
    D^{-1}E\\ -I_n
\end{smallmatrix}\right)\beta\\
&=\alpha^{\rm tr}(E^{\rm tr}-E-DB)\beta\\
&=\lr{\beta,\alpha}-\lr{\alpha,\beta}-\alpha^{\rm tr}DB\beta.
\end{flalign*}
Since $DB=D(I-R)-D(I-R')=E^{\rm tr}-E$, we obtain
$$\alpha^{\rm tr}DB\beta=\lr{\beta,\alpha}-\lr{\alpha,\beta}.$$
Thus, we finish the proof.
\end{proof}

The {\em $q$-twisted Hall algebra} $\mathcal {H}_{q}(\A)$ is defined to be the same space as $\mathcal {H}(\A)$, but with the twisted multiplication defined by
$$u_M\star u_N=q^{\lr{M,N}}u_M\diamond u_N.$$
\begin{proposition}\label{alg-hom-cdx0}
The map $\Psi:\mathcal {H}_{q}(\A)\longrightarrow\mathcal{T}_{q,\Lambda_1}$ defined on basis elements by $u_M\mapsto X_M$ is a homomorphism of algebras.
\end{proposition}
\begin{proof}
By Lemma \ref{zbyl1}, we obtain that $\mathcal {H}_{q}(\A)$ is the same as $\mathcal {H}_{\Lambda_1}(\A)$. Then we finish the proof by Theorem \ref{alg-hom-cdx1}.
\end{proof}

\begin{theorem}\label{dyzhu}
    For any $1\leq i\neq j\leq n$ and $\varepsilon=\pm 1$, in the quantum torus $\mathcal{T}_{q,\Lambda_1}$,
    \begin{itemize}
        \item[(1)] if $b_{ij}\leq 0$, then for any positive integers $l,p$ with $p\geq-lb_{ij}$,
        \begin{equation}
           \sum_{t=0}^{p+1}(-1)^tv_{i}^{(\varepsilon p+\varepsilon lb_{ij}-lb_{ij})t} \qbinom{p+1}{t}_{{v}_i} y_i^{p+1-t} y_j^l  y_i^{t}=0;
        \end{equation}
           \item[(2)] if $b_{ij}>0$, then for any positive integers $l,p$ with $p\geq lb_{ij}$,
           \begin{equation}
               \sum_{t=0}^{p+1}(-1)^tv_{i}^{(\varepsilon p-\varepsilon lb_{ij}-lb_{ij})t} \qbinom{p+1}{t}_{{v}_i} y_i^{p+1-t} y_j^l  y_i^{t}=0.
           \end{equation}
    \end{itemize}
\end{theorem}
\begin{proof}
If~$b_{ij}\leq 0$, then $c_{ij}=b_{ij}$, $\lr{{\bf e}_i,{\bf e}_j}=0$ and $\lr{{\bf e}_j,{\bf e}_i}=d_ib_{ij}$;~if~$b_{ij}> 0$, then $c_{ij}=-b_{ij}$, $\lr{{\bf e}_j,{\bf e}_i}=0$ and $\lr{{\bf e}_i,{\bf e}_j}=-d_ib_{ij}$. Using Proposition \ref{eq:twist-q-serre-rel-qca} and Lemma \ref{zbyl1}, we finish the proof.
\end{proof}
\begin{remark}
Let $\mathfrak{q}$ be an indeterminate. For any positive integer $m$ and $0\leq d\leq m$, set
\begin{equation*}
\llbracket m\rrbracket=\frac{\mathfrak{q}^m-1}{\mathfrak{q}-1},~\llbracket m\rrbracket^!=\prod\limits_{t=1}^m\llbracket t\rrbracket,~\llbracket\begin{smallmatrix}
m\\d
\end{smallmatrix}\rrbracket=\frac{\llbracket m\rrbracket^!}{\llbracket d\rrbracket^!\llbracket m-d\rrbracket^!}.
\end{equation*}
Then \begin{equation}\label{fangkh}\llbracket\begin{smallmatrix}
m\\d
\end{smallmatrix}\rrbracket_{q_i}=v_i^{d(m-d)}[ \begin{smallmatrix}
m\\d
\end{smallmatrix}]_{v_i}.\end{equation}
\end{remark}
\begin{corollary}\label{jibgx}
    For any $1\leq i\neq j\leq n$ and $\varepsilon=\pm 1$, in the quantum torus $\mathcal{T}_{q,\Lambda_1}$,
    \begin{itemize}
        \item[(1)] if $b_{ij}\leq 0$, then for any positive integers $l,p$ with $p\geq-lb_{ij}$,
        \begin{equation}\label{eqn:fundamental-rel-negative0}
           \sum_{t=0}^{p+1}(-1)^tv_{i}^{t(t-1)+(\varepsilon-1)(p+lb_{ij})t}\qqbinom{p+1}{t}_{q_i} y_i^{p+1-t} y_j^l  y_i^{t}=0;
        \end{equation}
           \item[(2)] if $b_{ij}>0$, then for any positive integers $l,p$ with $p\geq lb_{ij}$,
           \begin{equation}\label{eqn:fundamental-rel-positive0}
           \sum_{t=0}^{p+1}(-1)^tv_{i}^{((\varepsilon-1)p-(\varepsilon+1)lb_{ij})t+t(t-1)}\qqbinom{p+1}{t}_{q_i} y_i^{p+1-t} y_j^l  y_i^{t}=0.
        \end{equation}
    \end{itemize}
\end{corollary}
\begin{proof}
By Theorem \ref{dyzhu} and the equation $(\ref{fangkh})$, we finish the proof.
\end{proof}
\begin{remark}
The equation $(\ref{eqn:fundamental-rel-negative0})$ for $\varepsilon=1$ and the equation $(\ref{eqn:fundamental-rel-positive0})$ for $\varepsilon=-1$ has been given in
\cite[Theorem 4.2]{HCDX25}. The relations $(\ref{eqn:fundamental-rel-negative0})$ and $(\ref{eqn:fundamental-rel-positive0})$ are the so-called {\em high order fundamental relations}.
\end{remark}
As a special case, we have the following so-called {\em fundamental relations}.
\begin{corollary}\cite[Theorem 3.4]{HCDX25}
    For any $1\leq i\neq j\leq n$, in the quantum torus $\mathcal{T}_{q,\Lambda_1}$,
    \begin{itemize}
        \item[(1)] if $b_{ij}\leq 0$, we have
        \begin{equation}\label{eqn:fundamental-rel-negative1}
           \sum_{t=0}^{1-b_{ij}}(-1)^tq_{i}^{\frac{t(t-1)}{2}}\qqbinom{1-b_{ij}}{t}_{q_i} y_i^{1-b_{ij}-t} y_j  y_i^{t}=0;
        \end{equation}
           \item[(2)] if $b_{ij}>0$, we have
           \begin{equation}\label{eqn:fundamental-rel-positive1}
           \sum_{t=0}^{1+b_{ij}}(-1)^tq_{i}^{\frac{t(t-1)}{2}-tb_{ij}}\qqbinom{1+b_{ij}}{t}_{q_i} y_i^{1+b_{ij}-t} y_j  y_i^{t}=0.
        \end{equation}
    \end{itemize}
\end{corollary}
\begin{proof}
Taking $(l,p)=(1,-b_{ij})$ and $(l,p)=(1,b_{ij})$, together with $\varepsilon=1$, in $(\ref{eqn:fundamental-rel-negative0})$ and $(\ref{eqn:fundamental-rel-positive0})$, respectively, we complete the proof.
\end{proof}

\section{Quantum Serre relations in quantum cluster algebras}
In this section, let ${Q}$ be an acyclic valued quiver with the vertex set
$\{1,\ldots,n\}$. Let $\widetilde{Q}$ be the valued quiver obtained from $Q$ by attaching additional vertices $n+1,\ldots, 2n$ to $Q$, adding the arrow $n+i\to i$, and setting the valuation of the $(n+i)$-th vertex to be $d_{n+i}=d_i$ for each $1\leq i\leq n$. As the subsection \ref{s3.1}, we have the matrices $R,R',\widetilde{R},\widetilde{R}',\widetilde{E},\widetilde{E}'$, and also $c_{ij}=-r_{ij}-r_{ij}'$ for any $1\leq i\neq j\leq n$.  Let $\mathcal{A}$ be the category of finite-dimensional $k$-representations of $Q$, and $E$ be the matrix of the Euler form of $\A$ under the standard basis. Since $\A$ is hereditary, it is well-known that $I_n-R$ is invertible in the matrix ring ${M}_{n\times n}(\mathbb{Z})$. That is, $D^{-1}E^{\rm tr}$ is invertible in the matrix ring ${M}_{n\times n}(\mathbb{Z})$.
\subsection{A special compatible pair}\label{s:speical-compatible pair}
Define the matrices
\begin{equation}
\widetilde{B}:=\left(
  \begin{array}{c}
    D^{-1}(E^{\rm tr}-E) \\
    I_n \\
  \end{array}
\right)=\left(
  \begin{array}{c}
    R'-R \\
    I_n \\
  \end{array}
\right)=\widetilde{ {R}}'-\widetilde{ {R}}\in  {M}_{2n\times n}(\mathbb{Z})
\end{equation}
and
\begin{equation}
\Gamma:=\left(
  \begin{array}{cc}
    E^{\rm tr}-E & -E^{\rm tr}-E\\
    E^{\rm tr}+E & E^{\rm tr}-E\\
  \end{array}
\right)\in  {M}_{2n\times 2n}(\mathbb{Z}).
\end{equation}
Set
\begin{equation}
\Theta:=\left(
   \begin{array}{cc}
    -D^{-1}E & -D^{-1}E^{\rm tr}\\
    I_n & 0\\
  \end{array}
\right)=\left(
   \begin{array}{cc}
    R'-I_n & R-I_n\\
    I_n & 0\\
  \end{array}
\right)\in  {M}_{2n\times 2n}(\mathbb{Z}).
\end{equation}
It is easy to obtain the invertible matrix of $\Theta$
\begin{equation}
\Theta^{-1}=\left(
   \begin{array}{cc}
    0 &I_n\\
    -(E^{\rm tr})^{-1}D & -(E^{\rm tr})^{-1}E\\
  \end{array}
\right)\in  {M}_{2n\times 2n}(\mathbb{Z}).
\end{equation}
Define $\Lambda_0:=(\Theta^{-1})^{\rm tr}\Gamma \Theta^{-1}\in  {M}_{2n\times 2n}(\mathbb{Z})$, and clearly it is skew-symmetric. Then we have the following.
\begin{proposition}\label{xrd}
${\widetilde{B}}^{\rm tr}\Lambda_0=(2D,0)$. That is, $(\Lambda_0,\widetilde{B})$ is a compatible pair.
\end{proposition}
\begin{proof} By direct calculations, we obtain
\begin{flalign*}
\Lambda_0=(\Theta^{-1})^{\rm tr}\Gamma \Theta^{-1}=\begin{pmatrix} \begin{smallmatrix}D(E^{-1}-(E^{\rm tr})^{-1})D\quad & -D(E^{-1}E^{\rm tr}+(E^{\rm tr})^{-1}E) \\ (E(E^{\rm tr})^{-1}+E^{\rm tr}E^{-1})D\quad & E^{\rm tr}-E+E(E^{\rm tr})^{-1}E-E^{\rm tr}E^{-1}E^{\rm tr} \end{smallmatrix}\end{pmatrix}
\end{flalign*}
and then can easily finish the proof.
\end{proof}

\subsection{From Ringel--Hall algebras to Quantum cluster algebras}\label{s:hall-to-qca}
In this subsection, we take $\Lambda_2=\frac{1}{2}\Lambda_0$. Then we have ${\widetilde{B}}^{\rm tr}\Lambda_2=(D,0)$. For this special compatible pair $(\Lambda_2, \widetilde{B})$, the corresponding quantum cluster algebra
$\mathcal{A}_{\mathfrak{q}}(\Lambda_2,\widetilde{B})$ is called the {\em quantum cluster algebra with principal coefficients}.

Let $\A_q(Q)$ be the quantum cluster algebra  $\mathcal{A}_{\mathfrak{q}}(\Lambda_2,\widetilde{B})$ with principal
coefficients, specialised at $\mathfrak{q}=q$.
According to \cite{Rupel2}, $\A_q(Q)$ is the $\mathbb{Q}\mathbb{P}$-subalgebra of $\mathcal {F}_{q,\Lambda_2}$ generated by
\begin{equation*}
\{x_i, X_{M}~|~1\leq i\leq n,~M\in\Ind(\A)~~\text{is rigid}\}.
\end{equation*}

\begin{lemma}\label{zhishu1}
For any $\alpha,\beta,\gamma,\delta\in\mathbb{Z}^n$, we have
$$\Gamma(\left(\begin{smallmatrix}
    \alpha\\ \beta
\end{smallmatrix}\right),\left(\begin{smallmatrix}
    \gamma\\ \delta
\end{smallmatrix}\right))=\lr{\gamma-\delta,\alpha}+\lr{\beta,\gamma-\delta}+\lr{\gamma+\delta,\beta}
-\lr{\alpha,\gamma+\delta}.$$
\end{lemma}
\begin{proof} By direct calculations, we have
\begin{flalign*}
&\Gamma(\left(\begin{smallmatrix}
    \alpha\\ \beta
\end{smallmatrix}\right),\left(\begin{smallmatrix}
    \gamma\\ \delta
\end{smallmatrix}\right))\\&=(\alpha^{\rm tr},\beta^{\rm tr})\left(\begin{smallmatrix}
    E^{\rm tr}-E~ & -E^{\rm tr}-E\\
    E^{\rm tr}+E~ & E^{\rm tr}-E\\
 \end{smallmatrix}\right)\left(\begin{smallmatrix}
    \gamma\\ \delta
\end{smallmatrix}\right)\\
 &=\alpha^{\rm tr}E^{\rm tr}\gamma-\alpha^{\rm tr}E\gamma+\beta^{\rm tr}E^{\rm tr}\gamma+\beta^{\rm tr}E\gamma-\alpha^{\rm tr}E^{\rm tr}\delta-\alpha^{\rm tr}E\delta
 +\beta^{\rm tr}E^{\rm tr}\delta-\beta^{\rm tr}E\delta\\
 &=\lr{\gamma,\alpha}-\lr{\alpha,\gamma}+\lr{\gamma,\beta}+\lr{\beta,\gamma}-\lr{\delta,\alpha}-\lr{\alpha,\delta}+\lr{\delta,\beta}-\lr{\beta,\delta}\\
 &=\lr{\gamma-\delta,\alpha}+\lr{\beta,\gamma-\delta}+\lr{\gamma+\delta,\beta}
-\lr{\alpha,\gamma+\delta}.
\end{flalign*}
\end{proof}
\begin{lemma}\label{zhishu}
For any $\alpha,\beta\in\mathbb{Z}^n$, we have

$(1)$~$\Lambda(\widetilde{E}\alpha,\widetilde{E}'\beta)=-\frac{1}{2}(\alpha,\beta)$;\quad
$(2)$~$\Lambda(\widetilde{E}\alpha,\widetilde{E}\beta)=\frac{1}{2}(\lr{\beta,\alpha}-\lr{\alpha,\beta})$.
\end{lemma}
\begin{proof}
Note that $\widetilde{E}={D^{-1}E\choose -I_n}$ and $\widetilde{E}'={D^{-1}E^{\rm tr}\choose 0}$.

$(1)$~By definitions, we have
\begin{flalign*}
&\Lambda(\widetilde{E}\alpha,\widetilde{E}'\beta)\\&=\frac{1}{2}\alpha^{\rm tr}(E^{\rm tr}D^{-1},-I_n)\left( \begin{smallmatrix}
    0~& -DE^{-1}\\
    I_n~& -E^{\rm tr}E^{-1}\\
 \end{smallmatrix}\right)\Gamma\left(\begin{smallmatrix}
    0~& I_n\\
    -(E^{\rm tr})^{-1}D~& -(E^{\rm tr})^{-1}E\\
 \end{smallmatrix}\right)\left(\begin{smallmatrix}
    D^{-1}E^{\rm tr}\\ 0
\end{smallmatrix}\right)\beta\\
 &=\frac{1}{2}\alpha^{\rm tr}(-I_n,0)\Gamma\left(\begin{smallmatrix}
    0\\ -I_n
\end{smallmatrix}\right)\beta\\
 &=\frac{1}{2}\Gamma(\left(\begin{smallmatrix}
    -\alpha\\ 0
\end{smallmatrix}\right),\left(\begin{smallmatrix}
    0\\ -\beta
\end{smallmatrix}\right))\\
 &=-\frac{1}{2}(\alpha,\beta).
\end{flalign*}

$(2)$~Using $\widetilde{E}'=\widetilde{E}+\widetilde{B}$ and \cite[Lemma 3.1]{DSC}, we obtain
\begin{flalign*}\Lambda(\widetilde{E}\alpha,\widetilde{E}\beta)&=\Lambda(\widetilde{E}\alpha,\widetilde{E}'\beta)-\Lambda(\widetilde{E}\alpha,\widetilde{B}\beta)\\
&=-\frac{1}{2}(\alpha,\beta)+\lr{\beta,\alpha}\\
&=\frac{1}{2}(\lr{\beta,\alpha}-\lr{\alpha,\beta}).
\end{flalign*}
\end{proof}

\begin{lemma}
For any $M,N\in\A$, in $\H_{\Lambda_2}(\A)$ we have
\begin{flalign}u_{M}\star u_{N}=v^{\frac{1}{2}\lr{{\bf n},{\bf m}}+\frac{3}{2}\lr{{\bf m},{\bf n}}}\sum_{[L]}\frac{|\mathrm{Ext}_{\A}^{1}(M,N)_{L}|}{|\mathrm{Hom}_{\A}(M,N)|}u_L.
\end{flalign}
\end{lemma}
\begin{proof}
It is proved by the equation $(\ref{lmdc})$ and Lemma $\ref{zhishu}(2)$.
\end{proof}

\begin{proposition}\label{zhudl}
There exists an isomorphism of algebras $\rho:\mathcal {H}_{v}(\A)\rightarrow\mathcal {H}_{\Lambda_2}(\A)$ defined on basis elements by
$$\rho(u_M)=v^{\frac{1}{2}\lr{{\bf m},{\bf m}}}u_M.$$
\end{proposition}
\begin{proof}
For any  $M,N\in\A$, we have
\begin{flalign*}
&\rho(u_M)\star \rho(u_N)=v^{\frac{1}{2}\lr{{\bf m},{\bf m}}+\frac{1}{2}\lr{{\bf n},{\bf n}}}u_M\star u_N\\
&=v^{\frac{1}{2}\lr{{\bf m},{\bf m}}+\frac{1}{2}\lr{{\bf n},{\bf n}}+\frac{1}{2}\lr{{\bf n},{\bf m}}+\frac{3}{2}\lr{{\bf m},{\bf n}}}\sum_{[L]}\frac{|\mathrm{Ext}_{\A}^{1}(M,N)_{L}|}{|\mathrm{Hom}_{\A}(M,N)|}u_L\\
&=v^{\frac{1}{2}\lr{{\bf m},{\bf m}}+\frac{1}{2}\lr{{\bf n},{\bf n}}+\frac{1}{2}\lr{{\bf n},{\bf m}}+\frac{3}{2}\lr{{\bf m},{\bf n}}-\frac{1}{2}\lr{{\bf m}+{\bf n},{\bf m}+{\bf n}}}\sum_{[L]}\frac{|\mathrm{Ext}_{\A}^{1}(M,N)_{L}|}{|\mathrm{Hom}_{\A}(M,N)|}\rho(u_L)\\
&=v^{\lr{{\bf m},{\bf n}}}\sum_{[L]}\frac{|\mathrm{Ext}_{\A}^{1}(M,N)_{L}|}{|\mathrm{Hom}_{\A}(M,N)|}\rho(u_L)\\
&=\rho(u_M\ast u_N).
\end{flalign*}
Since $\rho$ sends the basis of $\mathcal {H}_{v}(\A)$ to the basis of $\H_{\Lambda_2}(\A)$, we obtain $\rho$ is an isomorphism.
\end{proof}

\begin{corollary}\label{qytl}
There exists an algebra homomorphism $\varphi:\mathcal {H}_{v}(\A)\rightarrow\mathcal{T}_{\Lambda_2,q}$ defined on basis elements by
$$\varphi(u_M)=v^{\frac{1}{2}\lr{{\bf m},{\bf m}}}X_M.$$
\end{corollary}
\begin{proof}
It is proved by Theorem \ref{alg-hom-cdx1} and Proposition \ref{zhudl}.
\end{proof}

\begin{corollary}\label{sy1}
There exists an algebra homomorphism $\varphi:\mathbf{U}^+_v(\mathfrak{g})\rightarrow\A_q(Q)$ defined on generators by
$$\varphi(E_i)=v_i^{\frac{1}{2}}(q_i-1)^{-1}y_i.$$
\end{corollary}
\begin{proof}
It is proved by Theorem \ref{ringel}, Corollary \ref{qytl} and Proposition \ref{sy2}.
\end{proof}

\begin{corollary}\label{serre}
For any $1\leq i\neq j\leq n$, we have the quantum Serre relations
$$\sum_{t=0}^{1-c_{ij}} (-1)^t \qbinom{1-c_{ij}}{t}_{{v}_i} y_i^{1-c_{ij}-t} y_j y_i^{t}=0$$
in the quantum cluster algebra $\A_q(Q)$, where $-c_{ij}=[b_{ij}]_{+}+[-b_{ij}]_{+}$.
\end{corollary}
\begin{proof}
It is proved by Corollary \ref{sy1}.
\end{proof}

\begin{corollary}\label{cor:high}
Let $1\leq i\neq j\leq n$ and $\varepsilon=\pm 1$. For any positive integers $l,p$ with $p\geq-lc_{ij}$, we have
$$\sum_{t=0}^{p+1} (-1)^t v_i^{\varepsilon(p+lc_{ij})t}\qbinom{p+1}{t}_{{v}_i} y_i^{p+1-t} y_j^l  y_i^{t}=0$$
in the quantum cluster algebra $\A_q(Q)$, where $-c_{ij}=[b_{ij}]_{+}+[-b_{ij}]_{+}$.
\end{corollary}
\begin{proof}
It is proved by Theorem \ref{high} and Corollary \ref{qytl}.
\end{proof}

In the end, let us give an example to illustrate the fundamental relations and quantum Serre relations. In the following, we write $[a,b]_t=ab-tba$ for any $a,b\in\mathcal{T}_{\Lambda,q}$ and $t\in\mathbb{Q}(v)$.
\begin{example}
Let $Q$ be the quiver $1\longrightarrow 2$ and $D=I_2$. Then $\widetilde{B}=\begin{pmatrix}0&1\\-1&0\\1&0\\0&1\end{pmatrix}$,
$R=\begin{pmatrix}0&0\\1&0\end{pmatrix}$ and $E=I_2-R^{\rm tr}=\begin{pmatrix}1&-1\\0&1\end{pmatrix}$. Thus, we have
$$\Lambda_1=\begin{pmatrix}0&0&-1&0\\0&0&0&-1\\1&0&0&-1\\0&1&1&0\end{pmatrix},~\Theta=\begin{pmatrix}-1&1&-1&0\\0&-1&1&-1\\1&0&0&0\\0&1&0&0\end{pmatrix}~\text{and}~\Gamma=\begin{pmatrix}0&1&-2&1\\-1&0&1&-2\\2&-1&0&1\\-1&2&-1&0\end{pmatrix}.$$
By direct calculations, we obtain $$\Lambda_2=\frac{1}{2}\Lambda_0=\frac{1}{2}(\Theta^{-1})^{\rm tr}\Gamma \Theta^{-1}=\frac{1}{2}\begin{pmatrix}0&1&-1&0\\-1&0&0&-1\\1&0&0&-1\\0&1&1&0\end{pmatrix}.$$
By the mutation formula $(\ref{mutation})$, we have
\begin{flalign*}y_1=X^{e_3-e_1}+X^{e_2-e_1}\quad\text{and}\quad y_2=X^{e_1+e_4-e_2}+X^{-e_2}.\end{flalign*}

In $\mathcal{T}_{q,\Lambda_1}$, we have
\begin{flalign*}\sum_{t=0}^{2} (-1)^t v^{-t}\qbinom{2}{t}_{{v}} y_1^{2-t} y_2 y_1^{t}&=[y_1,[y_1,y_2]]_{q^{-1}}\\
&=[X^{e_3-e_1}+X^{e_2-e_1},(v^{-1}-v)X^{e_4}]_{q^{-1}}\\
&=0,\end{flalign*}
which is the fundamental relation.

Now, let us check the high order fundamental relations (with $\varepsilon=1$, $l=2$ and $p=2$) in Theorem \ref{dyzhu}. Note that
\begin{flalign*}\sum_{t=0}^{3} (-1)^t q^{-t}\qbinom{3}{t}_{{v}} y_1^{3-t} y_2^2 y_1^{t}=[y_1,[y_1,[y_1,y_2^2]_{q^{-2}}]_{q^{-1}}].\end{flalign*}

By direct calculations, in $\mathcal{T}_{q,\Lambda_1}$ we have
\begin{flalign*}y_2^2&=(X^{e_1+e_4-e_2}+X^{-e_2})^2\\&=X^{2e_1+2e_4-2e_2}+(v+v^{-1})X^{e_1+e_4-2e_2}+X^{-2e_2}\end{flalign*}
and then
\begin{flalign*}[y_1,y_2^2]_{q^{-2}}=(1-q^{-2})(X^{e_1+e_3+2e_4-2e_2}+(v+v^{-1})X^{e_3+e_4-2e_2}+X^{e_3-e_1-2e_2}+X^{e_4-e_2}+X^{-e_1-e_2}).\end{flalign*}
Thus, in $\mathcal{T}_{q,\Lambda_1}$ we have
\begin{flalign*}[y_1,[y_1,y_2^2]_{q^{-2}}]_{q^{-1}}&=(1-q^{-1})(1-q^{-2})(X^{2e_3+2e_4-2e_2}+(v+v^{-1})X^{2e_3+e_4-e_1-2e_2}+X^{2e_3-2e_1-2e_2}+\\
&~~\quad(v+v^{-1})X^{e_3+e_4-e_1-e_2}+(v+v^{-1})X^{e_3-2e_1-e_2}+X^{-2e_1}).\end{flalign*}
Hence, in $\mathcal{T}_{q,\Lambda_1}$ we obtain
\begin{flalign*}[y_1,[y_1,[y_1,y_2^2]_{q^{-2}}]_{q^{-1}}]&=(1-q^{-1})(1-q^{-2})(q-q^{-1})(X^{2e_3+e_4-2e_1-e_2}+X^{2e_3-3e_1-e_2}\\&~~\quad-X^{2e_3+e_4-2e_1-e_2}
-X^{2e_3-3e_1-e_2}+X^{e_3-3e_1}-X^{e_3-3e_1})\\
&=0.\end{flalign*}

In $\mathcal{T}_{q,\Lambda_2}$, we have
\begin{flalign*}\sum_{t=0}^{2} (-1)^t \qbinom{2}{t}_{{v}} y_1^{2-t} y_2 y_1^{t}&=[y_1,[y_1,y_2]_v]_{v^{-1}}\\
&=[X^{e_3-e_1}+X^{e_2-e_1},(v^{-1}-v)v^{\frac{1}{2}}X^{e_4}]_{v^{-1}}\\&=0,\end{flalign*}
which is the quantum Serre relation.

Now, let us check the high order quantum Serre relations (with $l=2$ and $p=2$) in Corollary \ref{cor:high}. Note that
\begin{flalign*}\sum_{t=0}^{3} (-1)^t \qbinom{3}{t}_{{v}} y_1^{3-t} y_2^2 y_1^{t}=[y_1,[y_1,[y_1,y_2^2]_{q}]]_{q^{-1}}.\end{flalign*}
By direct calculations, in $\mathcal{T}_{q,\Lambda_2}$ we have
\begin{flalign*}y_2^2&=(X^{e_1+e_4-2_2}+X^{-e_2})^2\\&=X^{2e_1+2e_4-2e_2}+(v+v^{-1})X^{e_1+e_4-2e_2}+X^{-2e_2}\end{flalign*}
and then
\begin{flalign*}[y_1,y_2^2]_{q}=v(q^{-1}-q)(X^{e_1+2e_4-e_2}+X^{e_4-e_2}).\end{flalign*}
Thus, in $\mathcal{T}_{q,\Lambda_2}$ we have
\begin{flalign*}[y_1,[y_1,y_2^2]_{q}]=(q^{-1}-1)(1-q^2)X^{2e_4}.\end{flalign*}
Hence, in $\mathcal{T}_{q,\Lambda_2}$ we obtain
\begin{flalign*}[y_1,[y_1,[y_1,y_2^2]_{q}]]_{q^{-1}}=(q^{-1}-1)(1-q^2)[X^{e_3-e_1}+X^{e_2-e_1},X^{2e_4}]_{q^{-1}}=0.\end{flalign*}
\end{example}


\end{document}